\documentclass[a4paper,12pt]{amsart}
\input{preamble.tex}
\title{Notes on the coexistence of limit notions}
\author{Takashi Yamazoe}
\address{Graduate School of System Informatics, Kobe University,
	Rokko--dai 1--1, Nada--ku, 657--8501 Kobe, Japan}
\email{212x502x@cloud.kobe-u.jp}
\begin{document}

	\begin{abstract}
	We summarize the current knowledge on the three limit notions: ultrafilter-limits, closed-ultrafilter-limits and FAM-limits. Also, we consider the possibility to perform an iteration which has all the three limits and clarify the problem we face. 
\end{abstract}
	\maketitle
	
%

	\section{Introduction}\label{sec_int}
	\subsection{Cicho\'n's maximum}
Cicho\'n's maximum is a model where all the cardinal invariants in Cicho\'n's diagram (except for the two dependent numbers $\addm$ and $\cofm$) are pairwise different. Goldstern, Kellner and Shelah \cite{GKS} proved that $\aleph_1<\addn<\covn<\bb<\nonm<\covm<\dd<\nonn<\cofn<2^{\aleph_0}$ consistently holds assuming four strongly compact cardinals.
Later, they and Mej\'{\i}a \cite{GKMS} eliminated the large cardinal assumption and hence proved that Cicho\'n's maximum is consistent with ZFC.
We study what kinds of cardinal invariants can be added to Cicho\'n's maximum with distinct values. In this article, we deal with the following numbers:


\begin{dfn}
	\begin{itemize}
		
		\item A pair $\pi=(D,\{\pi_n:n\in D\})$ is a predictor if $D\in\ooo$ and each $\pi_n$ is a function $\pi_n\colon\omega^n\to\omega$. $Pred$ denotes the set of all predictors.
		\item $\pi\in Pred$ predicts $f\in\oo$ if $f(n)=\pi_n(f\on n)$ for all but finitely many $n\in D$. 
		$f$ evades $\pi$ if $\pi$ does not predict $f$.
		\item The prediction number $\pre$ and the evasion number $\ee$ are defined as follows\footnote{While the name ``prediction number'' and the notation ``$\pre$'' are not common, we use them in this article.}:
		\begin{gather*}
			\mathfrak{pr}\coloneq\min\{|\Pi|:\Pi\subseteq Pred,\forall f\in\oo~\exists\pi\in \Pi~ \pi \text{ predicts } f\},\\
			\mathfrak{e}\coloneq\min\{|F|:F\subseteq\oo,\forall \pi\in Pred~\exists f\in F~ f\text{ evades }\pi\}.
		\end{gather*}
		
	\end{itemize}
\end{dfn}
There are variants of the evasion/prediction numbers.
\begin{dfn}
	\label{dfn_variant}
	\begin{enumerate}
		\item A predictor $\pi$ \textit{bounding-predicts} $f\in\oo$ if $f(n)\leq\pi(f\on n)$ for all but finitely many $n\in D$.
		$\preb$ and $\eeb$ denote the prediction/evasion number respectively with respect to the bounding-prediction.
		\item Let $g\in\left(\omega+1\setminus2\right)^\omega$. ( ``$\setminus2$'' is required to exclude trivial cases.)
			\textit{$g$-prediction} is the prediction where the range of functions $f$ is restricted to $\prod_{n<\omega}g(n)$ and $\pre_g$ and $\ee_g$ denote the prediction/evasion number respectively with respect to the $g$-prediction.
			Namely,
			\begin{gather*}
				\mathfrak{pr}_g\coloneq\min\{|\Pi|:\Pi\subseteq Pred,\forall f\in\textstyle{\prod_{n<\omega}g(n)}~\exists\pi\in \Pi~\pi \text{ predicts } f\},\\
				\mathfrak{e}_g\coloneq\min\{|F|:F\subseteq\textstyle{\prod_{n<\omega}g(n)},\forall \pi\in Pred~\exists f\in F~ f\text{ evades }\pi\}.
			\end{gather*}
			Define:
			\begin{gather*}
				\mathfrak{pr}_{ubd}\coloneq\sup\left\{\mathfrak{pr}_g:g\in\left(\omega\setminus2\right)^\omega\right\},\\
				\ee_{ubd}\coloneq\min\left\{\ee_g:g\in\left(\omega\setminus2\right)^\omega\right\}.
			\end{gather*}

		\end{enumerate}
		
\end{dfn}
$\mathcal{E}$ denotes the $\sigma$-ideal generated by closed null sets and we define the four numbers $\add(\mathcal{E}),\non(\mathcal{E}),\cov(\mathcal{E})$ and $\cof(\mathcal{E})$ in the usual way. 

The above cardinal invariants can be added to Cicho\'n's diagram as in Figure \ref{fig_Cd_E}.
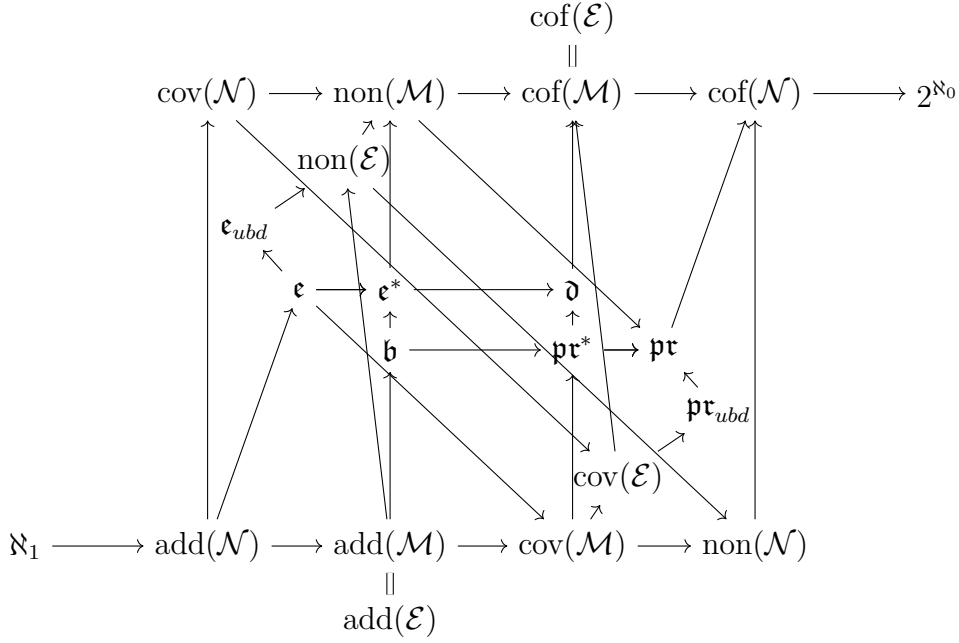
\begin{figure}
	\centering
	\begin{tikzpicture}
		\tikzset{
			textnode/.style={text=black}, 
		}
		\tikzset{
			edge/.style={color=black, thin}, 
		}
		\tikzset{cross/.style={preaction={-,draw=white,line width=7pt}}}
		\newcommand{\w}{2.4}
		\newcommand{\h}{2.0}
		
		\node[textnode] (addN) at (0,  0) {$\addn$};
		\node[textnode] (covN) at (0,  \h*3) {$\covn$};

		\node[textnode] (addM) at (\w,  0) {$\addm$};
		\node[textnode] (addE) at (\w,  -0.5*\h) {$\add(\mathcal{E})$};
		
		\node[textnode] (b) at (\w,  1.3*\h) {$\bb$};
		\node[textnode] (nonM) at (\w,  \h*3) {$\nonm$};
		
		\node[textnode] (covM) at (\w*2,  0) {$\covm$};
		\node[textnode] (d) at (\w*2,  1.7*\h) {$\dd$};
		\node[textnode] (cofM) at (\w*2,  \h*3) {$\cofm$};
		\node[textnode] (cofE) at (\w*2,  3.5*\h) {$\cof(\mathcal{E})$};

		\node[textnode] (nonN) at (\w*3,  0) {$\nonn$};
		\node[textnode] (cofN) at (\w*3,  \h*3) {$\cofn$};
		
		\node[textnode] (aleph1) at (-\w,  0) {$\aleph_1$};
		\node[textnode] (c) at (\w*4,  \h*3) {$2^{\aleph_0}$};
		
		\node[textnode] (e) at (0.5*\w,  1.7*\h) {$\mathfrak{e}$}; 
		\node[textnode] (pr) at (2.5*\w,  1.3*\h) {$\mathfrak{pr}$}; 
		
		\node[textnode] (estar) at (\w,  1.7*\h) {$\ee^*$};
		\node[textnode] (prstar) at (\w*2,  1.3*\h) {$\mathfrak{pr}^*$};

		\node[textnode] (eubd) at (\w*0.2,  \h*2.1) {$\ee_{ubd}$};
		\node[textnode] (prubd) at (\w*2.8,  \h*0.9) {$\mathfrak{pr}_{ubd}$};
		
		\node[textnode] (nonE) at (\w*0.75,  \h*2.55) {$\non(\mathcal{E})$};
		\node[textnode] (covE) at (\w*2.25,  \h*0.45) {$\cov(\mathcal{E})$};

		\draw[->, edge] (addM) to (nonE);
		\draw[->, edge] (covE) to (cofM);		
		\draw[->, edge] (covM) to (covE);		
		\draw[->, edge] (e) to (estar);
		\draw[->, edge] (e) to (covM);
		\draw[->, edge] (addM) to (b);
		\draw[->, edge] (e) to (estar);
		\draw[->, edge] (b) to (prstar);
		\draw[->, edge] (nonM) to (pr);
		\draw[->, edge] (prstar) to (pr);
		
		\draw[->, edge] (addN) to (covN);
		\draw[->, edge] (addN) to (addM);
		\draw[->, edge] (covN) to (nonM);	
		
		\draw[->, edge] (addM) to (covM);
		\draw[->, edge] (nonM) to (cofM);
		\draw[->, edge] (d) to (cofM);
		
		\draw[->, edge] (covM) to (nonN);
		\draw[->, edge] (cofM) to (cofN);
		\draw[->, edge] (nonN) to (cofN);
		\draw[->, edge] (aleph1) to (addN);
		\draw[->, edge] (cofN) to (c);
		
		
		\draw[->, edge] (addN) to (e);
		
		\draw[->, edge] (covM) to (prstar);
		
		\draw[->, edge] (pr) to (cofN);

		\draw[->, edge] (b) to (estar);C
		\draw[->, edge] (estar) to (nonM);
		\draw[->, edge] (estar) to (d);
		\draw[->, edge] (e) to (eubd);
		
		\draw[->, edge] (prstar) to (d);
		\draw[->, edge] (prstar) to (pr);

		\draw[->, edge] (prubd) to (pr);

		\draw[->, edge,cross] (eubd) to (nonE);
		\draw[->, edge] (nonE) to (nonM);
		
		\draw[->, edge] (covE) to (prubd);
		\draw[->, edge] (nonE) to (nonN);
		\draw[->, edge] (covN) to (covE);

		\draw[double distance=2pt, edge] (addM) to (addE);
		\draw[double distance=2pt, edge] (cofM) to (cofE);

	\end{tikzpicture}
	\caption{Cicho\'n's diagram and other cardinal invariants.}\label{fig_Cd_E}
\end{figure}
	\subsection{Limit notions}
The construction of Cicho\'n's maximum consists of two steps: the first one is to separate the left side of the diagram with additional properties and the second one is to separate the right side (point-symmetrically) using these properties. In \cite{GKS}, the large cardinal assumption was used in the second step to apply \textit{Boolean Ultrapowers}. In \cite{GKMS}, they introduced the \textit{submodel} method instead, which is a general technique to separate the right side without large cardinals.

Let us focus on the first step. The main work to separate the left side is \textit{to keep the bounding number $\bb$ small} through the forcing iteration, since the naive bookkeeping iteration to increase the cardinal invariants in the left side guarantees the smallness of the other numbers but not of $\bb$.
To tackle the problem, in \cite{GKS} they used the \textit{ultrafilter-limit} 
method, which was first introduced by Goldstern, Mej\'{\i}a and Shelah \cite{GMS16} to separate the left side of the diagram. 

\begin{thm}(\cite[Main Lemma 4.6]{GMS16})
	Ultrafilter-limits keep $\bb$ small.
\end{thm}

Recently, other kinds of limit notions have been studied.

After the first construction of Cicho\'n's maximum in \cite{GKS}, Kellner, Shelah and T{\u{a}}nasie \cite{KST} constructed\footnote{It was constructed under the same large cardinal assumption as \cite{GKS}. Later, the assumption was eliminated in \cite{GKMS} introducing the submodel method.} Cicho\'n's maximum for another order $\aleph_1<\addn<\bb<\covn<\nonm<\covm<\nonn<\dd<\cofn<2^{\aleph_0}$,  
introducing the \textit{FAM-limit}\footnote{While they did not use the name ``FAM-limit'' but ``strong FAM-limit for intervals'', we use ``FAM-limit'' in this article. 
Also, the original idea of the notion is from \cite{She00}.} method, which focuses on (and actually is short for) finitely additive measures on $\omega$ and keeps the bounding number $\bb$ small as the ultrafilter-limit does. Recently, Uribe-Zapata \cite{Uri} formalized the theory of the FAM-limits and he, Cardona and Mej\'{\i}a proved that the limits also control $\non(\mathcal{E})$:

\begin{thm}(\cite{Car23RIMS})
\label{thm_FAM_nonE}
FAM-limits keep $\non(\mathcal{E})$ small.
\end{thm}

The author introduced in \cite{Yam24} a new limit notion \textit{closed-ultrafilter-limit}, which is a specific kind of the ordinary ultrafilter-limit and proved that it controls $\ee^*$:

\begin{thm}(\cite[Main Lemma 3.2.4]{Yam24})
	Closed-ultrafilter-limits keep $\eeb$ small.
\end{thm}

In \cite{Yam24}, the author also performed an iteration which has both ultrafilter-limits and closed-ultrafilter-limits and controlled both $\bb$ and $\eeb$. Thus, the following natural question arises:

\begin{ques}
	\label{ques_3_mix}
	Can we perform an iteration which has \textit{all the three limit notions}?
\end{ques}
If such an iteration is possible, we can control $\bb$, $\eeb$ and $\non(\mathcal{E})$ simultaneously and finally obtain a new separation constellation (see Figure \ref{fig_p7} in Section \ref{sec_E}).

However, at the moment we cannot answer Question \ref{ques_3_mix} positively, but we are close to achieving the construction.

In this article, we study the possibility of the construction. Also, this article is a detailed analysis of \cite[Question 5.4]{Yam24}.
	\subsection{Structure of the article}
In Section \ref{sec_RS_PT}, we review the relational systems, the Tukey order and the general preservation theory of fsi (finite support iteration), such as \textit{goodness}. In Section \ref{sec_uflimit}, we present the three limit notions: \textit{ultrafilter-limit}, \textit{closed-ultrafilter-limit} and \textit{FAM-limit}.
In Section \ref{sec_E}, we analyze the possibility to construct an iteration with all the three limit notions and clarify where the problem lies.
	
	\section{Relational systems and preservation theory} \label{sec_RS_PT}
	\begin{dfn}
	
	\begin{itemize}
		\item $\R=\langle X,Y,\sqsubset\rangle$ is a relational system if $X$ and $Y$ are non-empty sets and $\sqsubset \subseteq X\times Y$.
		\item We call an element of $X$ a \textit{challenge}, an element of $Y$ a \textit{response}, and ``$x\sqsubset y$''  ``$x$ is \textit{responded by }$y$''.
		\item $F\subseteq X$ is $\R$-unbounded if no response responds all challenges in $X$.
		\item $F\subseteq Y$ is $\R$-dominating if any challenge is responded by some challenge in $Y$.
		\item $\R$ is non-trivial if $X$ is $\R$-unbounded and $Y$ is $\R$-dominating. For non-trivial $\R$, define
		\begin{itemize}
			\item $\bb(\R)\coloneq\min\{|F|:F\subseteq X \text{ is }\R\text{ -unbounded}\}$, and
			\item $\dd(\R)\coloneq\min\{|F|:F\subseteq Y \text{ is }\R\text{ -dominating}\}$.
		\end{itemize}

	\end{itemize}
	
\end{dfn}
In this section, we assume $\R$ is non-trivial.

\begin{exa}
	\label{exa_RS}
	\begin{itemize}
		\item $\mathbf{D}\coloneq\langle \oo,\oo,\leq^*\rangle$. Note $\bb(\mathbf{D})=\bb, \dd(\mathbf{D})=\dd$.
		\item 
		$\mathbf{PR}\coloneq\langle\oo, Pred,\sqsubset^\mathrm{p}\rangle$, where $f\sqsubset^\mathrm{p}\pi\colon\Leftrightarrow f$ is predicted by $\pi$.
		Also, $\mathbf{BPR}\coloneq\langle\oo, Pred,\sqsubset^\mathrm{bp}\rangle$, where $f\sqsubset^\mathrm{bp}\pi\colon\Leftrightarrow f$ is bounding-predicted by $\pi$ and $\mathbf{PR}_g\coloneq\langle\prod_{n<\omega}g(n), Pred,\sqsubset^\mathrm{p}\rangle$  where $g\in(\omega+1\setminus2)^\omega$. Note $\bb(\mathbf{PR})=\ee, \dd(\mathbf{PR})=\pre$, $\bb(\mathbf{BPR})=\eeb, \dd(\mathbf{BPR})=\preb$, $\bb(\mathbf{PR}_g)=\ee_g, \dd(\mathbf{PR}_g)=\pre_g$.
		\item For an ideal $I$ on $X$, define two relational systems $\bar{I}\coloneq\langle I,I,\subseteq\rangle$ and $C_I\coloneq\langle X,I,\in\rangle$.
		Note $\bb(\bar{I})=\addi,\dd(\bar{I})=\cofi$ and $\bb(C_I)=\noni,~\dd(C_I)=\covi$.
		We write just $I$ instead of $\bar{I}$.
	\end{itemize}
	
\end{exa}

\begin{dfn}
	$\R^\bot$ denotes the dual of $\R=\langle X,Y,\sqsubset\rangle$,
	i.e.,
	$\R^\bot\coloneq\langle Y,X,\sqsubset^\bot\rangle$ where $y\sqsubset^\bot x:\Leftrightarrow \lnot(x\sqsubset y)$.
\end{dfn}

\begin{dfn}
	
	For relational systems $\R=\langle X,Y,\sqsubset \rangle, \R^{\prime}=\langle X^{\prime},Y^{\prime},\sqsubset^{\prime}~\rangle$,
	$(\Phi_-,\Phi_+):\R\rightarrow\R^\prime$ is a Tukey connection from $\R$ into $\R^{\prime}$ if $\Phi_-:X\rightarrow X^{\prime}$ and $\Phi_+:Y^{\prime}\rightarrow Y$ are functions such that:
	\begin{equation*}
		\forall x\in X~\forall y^{\prime}\in Y^{\prime}~\Phi_-(x)\sqsubset^{\prime} y^{\prime}\Rightarrow x \sqsubset \Phi_{+} (y^{\prime}).
	\end{equation*}

	We write $\R\preceq_T\R^{\prime}$ if there is a Tukey connection from $\R$ into $\R^{\prime}$ and call $\preceq_T$ the Tukey order.
	Tukey equivalence is defined by $\R\cong_T\R^{\prime}$ iff $\R\preceq_T\R^{\prime}$ and $\R^{\prime}\preceq_T\R$ hold. 
	
\end{dfn}
	
\begin{fac}
	\label{Tukey order and b and d}
	\begin{enumerate}
		\item $\R\preceq_T\R^{\prime}$ implies $(\R^{\prime})^\bot\preceq_T(\R)^\bot$.
		\item $\R\preceq_T\R^{\prime}$ implies $\mathfrak{b}(\R^{\prime})\leq\mathfrak{b}(\R)$ and $\mathfrak{d}(\R)\leq\mathfrak{d}(\R^{\prime})$.
	\end{enumerate}
\end{fac}

Using a relational system of the form $C_{[A]^{<\theta}}$
($[A]^{<\theta}$ is an ideal on $A$), $\bb(\R)$ and $\dd(\R)$ can be calculated from ``outside'' and ``inside'' .
\begin{cor}
	Let $\theta$ be regular uncountable and $A$ a set of size $\geq\theta$.
	\begin{itemize}
		\item[(outside)] If $\R\lq C_{[A]^{<\theta}}$, then $\theta\leq\bb(\R)$ and $\dd(\R)\leq|A|$.
		\item[(inside)] If $C_{[A]^{<\theta}}\lq \R$, then $\bb(\R)\leq\theta$ and $|A|\leq\dd(\R)$.
	\end{itemize}
\end{cor}

Both ``$\R\lq C_{[A]^{<\theta}}$'' and ``$C_{[A]^{<\theta}}\lq \R$'' have the following characterizations.
\begin{fac}(\cite[Lemma 1.16.]{forcing_constellations})
	\label{fac_Tukey_order_equivalence_condition}
	Let $\theta$ be an infinite cardinal, $I$ a (index) set of size $\geq\theta$ and $\R=\langle X,Y,\sqsubset\rangle$ a relational system.
	\begin{enumerate}
		\item If $|X|\geq\theta$, then $\R\lq C_{[X]^{<\theta}}$ iff any subset of $X$ of size $<\theta$ is $\R$-bounded (i.e., not $\R$-unbounded).
		\item \label{item_small_equiv}
		$C_{[I]^{<\theta}}\lq\R$ iff there exists $\langle x_i:i\in I\rangle$ such that every $y\in Y$ responds only $<\theta$-many $x_i$.
	\end{enumerate}
\end{fac}
%
\begin{fac}(\cite[Lemma 1.15.]{forcing_constellations})
	\label{fac_suff_eq_CI_and_I}
	If $\theta$ is regular and $A$ is a set with $|A|^{<\theta}=|A|$,
	then $C_{[A]^{<\theta}}\cong_T [A]^{<\theta}$.
\end{fac}

\begin{fac}(\cite[Lemma 2.11.]{forcing_constellations})
	\label{fac_cap_V}
	Let $\theta$ be uncountable regular and $A$ be a set of size $\geq\theta$.
	Then, any $\theta$-cc poset forces $[A]^{<\theta}\cong_T[A]^{<\theta}\cap V$ and $C_{[A]^{<\theta}}\cong_T C_{[A]^{<\theta}}\cap V$.
	Moreover, $\mathfrak{x}([A]^{<\theta})=\mathfrak{x}^V([A]^{<\theta})$ where $\mathfrak{x}$ represents ``$\add$'', ``$\cov$'', ``$\non$'' or ``$\cof$''.
\end{fac}

\begin{dfn}
	$\R=\langle X,Y,\sqsubset\rangle$ is a Polish relational system (Prs) if:
	\begin{enumerate}
		\item $X$ is a perfect Polish space.
		\item  $Y$ is analytic in a Polish space $Z$.
		\item $\sqsubset =\bigcup_{n<\omega}\sqsubset_n$ where $\langle \sqsubset_n:n<\omega\rangle$ is an ($\subseteq$-)increasing sequence of closed subsets of $X\times Z$ such that for any $n<\omega$ and any $y\in Y$, $\{x\in X:x\sqsubset_n y\}$ is closed nowhere dense.
	\end{enumerate}
	When dealing with a Prs, we interpret it depending on the model we are working in.
\end{dfn}

In this section, $\R=\langle X,Y,\sqsubset\rangle$ denotes a Prs. 
\begin{dfn}
	A poset $\p$ is $\theta$-$\R$-good if for any $\p$-name $\dot{y}$ for a member of $Y$, there is a non-empty set $Y_0\subseteq Y$ of size $<\theta$ such that for any $x\in X$, if $x$ is not responded by any $y\in Y_0$, then $\p$ forces $x$ is not responded by $\dot{y}$.
	If $\theta=\aleph_1$, we say ``$\R$-good'' instead of ``$\aleph_1$-$\R$-good''.
\end{dfn}
%
\begin{fac}(\cite{JS90}, \cite[Corollary 4.10.]{BCMseparating})
	Any fsi of $\theta$-cc $\theta$-$\R$-good posets is again $\theta$-$\R$-good whenever $\theta$ is regular uncountable.
\end{fac}

\begin{fac}(\cite{FM21},\cite[Theorem 4.11.]{BCMseparating})
	Let $\p$ be a fsi of non-trivial $\theta$-cc $\theta$-$\R$-good posets of length $\gamma\geq\theta$. 
	Then, $\p$ forces $C_{[\gamma]^{<\theta}}\lq\R$.
\end{fac}

\begin{fac}
	(\cite[Lemma 4]{Mej13},
	\cite[Theorem 6.4.7]{BJ95})
	If $\theta$ is regular, then any poset of size $<\theta$ is $\theta$-$\R$-good.
	In particular, Cohen forcing is $\R$-good.
\end{fac}

To treat goodness, we have to characterize cardinal invariants using a Prs.
While $\mathbf{D},\mathbf{PR},\mathbf{BPR}$ and $\mathbf{PR}_g$ are canonically Prs's, the cardinal invariants on ideals need other characterizations.

\begin{exa}
	\begin{enumerate}
		\item For $k<\omega$, let $\mathrm{id}^k\in\oo$ denote the function $i\mapsto i^k$ for each $i<\omega$ and let $\mathcal{H}\coloneq\{\mathrm{id}^{k+1}:k<\omega\}$.
		
		Let $\mathcal{S}=\mathcal{S}(\omega,\mathcal{H})$ be the set of all functions $\varphi\colon\omega\to\fin$ such that there is $h\in \mathcal{H}$ 
		with $|\varphi(i)|\leq h(i)$ for all $i<\omega$.
		Let $\mathbf{Lc}^*=\langle\oo,\mathcal{S},\in^*\rangle$ be the Prs where $x\in^*\varphi:\Leftrightarrow x(n)\in\varphi(n)$ for all but finitely many $n<\omega$.
		
		As a consequence of \cite{BartInv},
		$\mathbf{Lc}^*\cong_T\n$ holds.
		
		Any $\mu$-centered poset is $\mu^+$-$\mathbf{Lc}^*$-good (\cite{Bre91,JS90}).
		 
		Any Boolean algebra with a strictly positive finitely additive measure is $\mathbf{Lc}^*$-good (\cite{Kam89}). In particular, so is any subalgebra of random forcing.
		
		\item For each $n<\omega$,
		let $\Omega_n\coloneq\{a\in[\sq]^{<\omega}:\mathbf{Lb}_2(\bigcup_{s\in a}[s])\leq2^{-n}\}$ (endowed with the discrete topology) where $\mathbf{Lb}_2$ is the standard Lebesgue measure on $2^\omega$.
		Put $\Omega\coloneq\prod_{n<\omega}\Omega_n$ with the product topology, which is a perfect Polish space.
		For $x\in\Omega$, let $N_x^*\coloneq\bigcap_{n<\omega}\bigcup_{s\in x(n)}[s]$, a Borel null set in $2^\omega$.
		Define the Prs $\mathbf{Cn}\coloneq\langle\Omega,2^\omega,\sqsubset^{\mathbf{Cn}}\rangle$ where $x\sqsubset^\mathbf{Cn}z\colon\Leftrightarrow z\notin N^*_x$.
		Since $\langle N^*_x:x\in\Omega\rangle$ is cofinal in $\n(2^\omega)$ (the set of all null sets in $2^\omega$), $\mathbf{Cn}\cong_T C_\n^\bot$ holds. 
		
		Any $\mu$-centered poset is $\mu^+$-$\mathbf{Cn}$-good (\cite{Bre91}).
		
		\item Let $\Xi\coloneq\{f\in(\sq)^{\sq}:\forall s\in\sq, s\subseteq f(s)\}$ and define the Prs $\mathbf{Mg}\coloneq\langle2^\omega,\Xi,\in^\bullet\rangle$ where $x\in^\bullet f\colon\Leftrightarrow|\{s\in\sq:x\supseteq f(s)\}|<\omega$. Note that $\mathbf{Mg}\cong_T C_\m$.

	\end{enumerate}
\end{exa}
Summarizing the properties of the ``inside'' direction and the goodness, 
we obtain the following corollary, which will be actually applied to the iteration in Section \ref{sec_E}.
\begin{cor}
	\label{cor_smallness_for_addn_and_covn_and_nonm}
	Let $\theta$ be regular uncountable and $\p$ be a fsi of ccc forcings of length $\gamma\geq\theta$.
	
	\begin{enumerate}
		\item Assume that each iterand is either:
		\begin{itemize}
			\item of size $<\theta$,
			\item a subalgebra of random forcing, or
			\item $\sigma$-centered.
		\end{itemize}
		Then, $\p$ forces $C_{[\gamma]^{<\theta}}\lq\mathbf{Lc}^*$,
		in particular, $\addn\leq\theta$.
		\item  Assume that each iterand is either:
		\begin{itemize}
			\item of size $<\theta$, or
			\item $\sigma$-centered.
		\end{itemize}
		Then, $\p$ forces $C_{[\gamma]^{<\theta}}\lq\mathbf{Cn}$,
		in particular, $\covn\leq\theta$.
		\item Assume that each iterand is:
		\begin{itemize}
			\item of size $<\theta$.
		\end{itemize}
		Then, $\p$ forces $C_{[\gamma]^{<\theta}}\lq\mathbf{Mg}$,
		in particular, $\nonm\leq\theta$.
	\end{enumerate}

\end{cor}

	\section{Limit notions}\label{sec_uflimit}
	
	\subsection{Ultrafilter-limit and closed-ultrafilter-limit}
	We basically follow the presentation of \cite{Mej_Two_FAM} to describe the general theory of (closed-)ultrafilter-limits. Also, the original ideas are already in \cite{GMS16} and \cite{Yam24}. Thus, we often omit the proofs in this subsection. For details, see \cite{Yam24}.
\begin{dfn}(\cite[Section 5]{Mej19})\label{dfn_linked_class}
	Let $\Gamma$ be a class for subsets of posets,
	i.e., $\Gamma\in\prod_{\p}\mathcal{P}(\mathcal{P}(\p))$, a (class) function. (e.g., $\Gamma=\Lambda({\text{centered}})\coloneq$ ``centered'' is an example of a class for subsets of poset and in this case $\Gamma(\p)$ denotes the set of all centered subsets of $\p$ for each poset $\p$.)
	\begin{itemize}

	\item A poset $\p$ is $\mu$-$\Gamma$-linked if $\p$ is a union of $\leq\mu$-many subsets in $\Gamma(\p)$.
	As usual, when $\mu=\aleph_0$, we use ``$\sigma$-$\Gamma$-linked'' instead of ``$\aleph_0$-$\Gamma$-linked''.
	\item Abusing notation, we write ``$\Gamma\subseteq\Gamma^\prime$'' if $\Gamma(\p)\subseteq\Gamma^\prime(\p)$ holds for every poset $\p$.
    \end{itemize}
\end{dfn}


In this paper, an ``ultrafilter'' means a non-principal ultrafilter.
\begin{dfn}
	\label{dfn_UF_linked}
	
	Let $D$ be an ultrafilter on $\omega$ and $\p$ be a poset.
	\begin{enumerate}
		\item $Q\subseteq \p$ is $D$-lim-linked ($\in\Lambda^\mathrm{lim}_D(\p)$) if there exists a function $\lim^D\colon Q^\omega\to\p$ and a $\p$-name $\dot{D}^\prime$ of an ultrafilter extending $D$ such that for any countable sequence 
		$\bar{q}=\langle q_m:m<\omega\rangle\in Q^\omega$, 
		\begin{equation}
			\textstyle{\lim^D\bar{q}} \Vdash \{m<\omega:q_m \in \dot{G}\}\in \dot{D}^\prime.
		\end{equation}

		Moreover, if $\ran(\lim^D)\subseteq Q$, we say $Q$ is c-$D$-lim-linked (closed-$D$-lim-linked, $\in\Lambda^\mathrm{lim}_{\mathrm{c}D}(\p)$).
		\item $Q$ is (c-)uf-lim-linked (short for (closed-)ultrafilter-limit-linked) if $Q$ is (c-)$D$-lim-linked for every ultrafilter $D$.
		\item $\Lambda^\mathrm{lim}_\mathrm{uf}\coloneq\bigcap_D\Lambda^\mathrm{lim}_D$ and 
		$\Lambda^\mathrm{lim}_\mathrm{cuf}\coloneq\bigcap_D\Lambda^\mathrm{lim}_{\mathrm{c}D}$.

		
	\end{enumerate}

	We often say ``$\p$ has (c-)uf-limits'' instead of ``$\p$ is $\sigma$-(c-)uf-lim-linked''.

\end{dfn}

\begin{exa}
	Singletons are c-uf-lim-linked and hence every poset $\p$ is $|\p|$-c-uf-lim-linked. 
\end{exa}


\begin{dfn}
	\label{dfn_Gamma_iteration}
	\begin{itemize}
		\item A $\kappa$-$\Gamma$-iteration is a fsi $\p_\gamma=\langle(\p_\xi,\qd_\xi):\xi<\gamma\rangle $ with witnesses $\langle\p_\xi^-:\xi<\gamma\rangle$ and $\langle\dot{Q}_{\xi,\zeta}:\zeta<\theta_\xi,\xi<\gamma\rangle$ satisfying for all $\xi<\gamma$:
		\begin{enumerate}
			\item $\p^-_\xi\lessdot\p_\xi$.
			\item $\theta_\xi<\kappa$.
			\item \label{item_Q_is_P_minus_name}
			$\qd_\xi$ and $\langle\dot{Q}_{\xi,\zeta}:\zeta<\theta_\xi\rangle$ are $\p^-_\xi$-names and $\p^-_\xi$ forces that 
			$\bigcup_{\zeta<\theta_\xi}\dot{Q}_{\xi,\zeta}=\qd_\xi$ and $\dot{Q}_{\xi,\zeta}\in\Gamma(\qd_\xi)$ for each $\zeta<\theta_\xi$.
		\end{enumerate}
		\item $\xi<\gamma$ is a trivial stage if $\Vdash_{\p^-_\xi}|\dot{Q}_{\xi,\zeta}|=1$ for all $\zeta<\theta_\xi$. $S^-$ is the set of all trivial stages and $S^+\coloneq\gamma\setminus S^-$.
		\item A guardrail for the iteration is a function $h\in\prod_{\xi<\gamma}\theta_\xi$.
		\item $H\subseteq\prod_{\xi<\gamma}\theta_\xi$ is complete if any countable partial function in $\prod_{\xi<\gamma}\theta_\xi$ is extended to some (total) function in $H$.
		\item $\p^h_\eta$ is the set of conditions $p\in\p_\eta$ following $h$, i.e., for each $\xi\in\dom(p)$, $p(\xi)$ is a $\p^-_\xi$-name 
		and $\Vdash_{\p^-_\xi}p(\xi)\in \dot{Q}_{\xi,h(\xi)}$.  	
	\end{itemize}
\end{dfn}

%



\begin{lem}
	\label{cor_complete}
	If $\aleph_1\leq\mu\leq|\gamma|\leq2^\mu$ and $\mu^+=\kappa$, 
	then there exists a complete set of guardrails (of length $\gamma$) of size $\leq\mu^{\aleph_0}$.
\end{lem}

In this section, let $\Gamma_\mathrm{uf}$ represent $\Lambda^\mathrm{lim}_\mathrm{uf}$ or $\Lambda^\mathrm{lim}_\mathrm{cuf}$.
\begin{dfn}
	\label{dfn_UF_iteration}
	A $\kappa$-$\Gamma_\mathrm{uf}$-iteration has $\Gamma_\mathrm{uf}$-limits on $H$ if
	\begin{enumerate}
		\item $H\subseteq\prod_{\xi<\gamma}\theta_\xi$, a set of guardrails.
		\item For $h\in H$, $\langle \dot{D}^h_\xi:\xi\leq\gamma \rangle$ is a sequence such that $\dot{D}^h_\xi$ is a $\p_\xi$-name of a non-principal ultrafilter on $\omega$. 
		\item If $\xi<\eta\leq\gamma$, then $\Vdash_{\p_\eta}\dot{D}^h_\xi\subseteq\dot{D}^h_\eta$.
		
		\item \label{item_D^-}
		For $\xi<\gamma$, $\Vdash_{\p_\xi} (\dot{D}^h_\xi)^-\in V^{\p^-_\xi}$ where $(\dot{D}^h_\xi)^-\coloneq\dot{D}^h_\xi\cap V^{\p^-_\xi}$  if $ \xi\in S^+$, otherwise let $(\dot{D}^h_\xi)^-$ be an arbitrary ultrafilter in $V^{\p^-_\xi}$ (hence this item is trivially satisfied in this case).
		\item whenever $\langle \xi_m:m<\omega \rangle\subseteq \gamma$ and $\bar{q}=\langle \dot{q}_m:m<\omega\rangle$ satisfying 
		
		$\Vdash_{\p^-_{\xi_m}}\dot{q}_m\in\dot{Q}_{\xi_m,h(\xi_m)}$ for each $m<\omega$:
		\begin{enumerate}
			\item If $\langle \xi_m:m<\omega \rangle$ is constant with value $\xi$, then
			\begin{equation}
				\label{eq_constant}
				\Vdash_{\p_\xi}\textstyle{\lim^{(\dot{D}^h_\xi)^-}}\bar{q}\Vdash_{\qd_\xi}\{m<\omega:\dot{q}_m\in \dot{H}_\xi\}
				\in\dot{D}^h_{\xi+1}.
			\end{equation}
			($\dot{H}_\xi$ denotes the canonical name of $\qd_\xi$-generic filter over $V^{\p_\xi}$.)
			\item If $\langle \xi_m:m<\omega \rangle$ is increasing, then
			\begin{equation}
				\label{eq_increasing}
				\Vdash_{\p_\gamma}\{m<\omega:\dot{q}_m\in \dot{G}_\gamma\}\in\dot{D}^h_\gamma.
			\end{equation}
		\end{enumerate}
		
	\end{enumerate}
\end{dfn}

\begin{lem}
	\label{lem_uf_const_succ}
	Let $\p_{\gamma+1}$ be a $\kappa$-$\Gamma_\mathrm{uf}$-iteration (of length $\gamma+1$) and suppose $\p_\gamma=\p_{\gamma+1}\on\gamma$ has $\Gamma_\mathrm{uf}$-limits on $H$.
	If:
	\begin{equation}
		\label{eq_minus}
		\Vdash_{\p_\gamma} (\dot{D}^h_\gamma)^-\in V^{\p^-_\gamma}\text{ for all }h\in H,
	\end{equation}
	then we can find $\{\dot{D}^h_{\gamma+1}:h\in H\}$ witnessing that $\p_{\gamma+1}$ has $\Gamma_\mathrm{uf}$-limits on $H$.
\end{lem}

The limit step of the construction of ultrafilters is realized if we resort to centeredness.

\begin{lem}
	\label{lem_uf_const_all}
	Let $\gamma$ be limit and 
	$\p_\gamma$ be a $\kappa$-$\left(\Lambda(\mathrm{centered})\cap\Gamma_\mathrm{uf}\right) $-iteration.
	If $\langle \dot{D}^h_\xi:\xi<\gamma, h\in H\rangle$ witnesses that for any $\xi<\gamma$, $\p_\xi=\p_\gamma\on\xi$ has $\Gamma_\mathrm{uf}$-limits on $H$,  
	then we can find $\langle\dot{D}^h_\gamma:h\in H\rangle$ such that $\langle \dot{D}^h_\xi:\xi\leq\gamma, h\in H\rangle$ witnesses $\p_\gamma$ has $\Gamma_\mathrm{uf}$-limits on $H$.
\end{lem}
We give a proof since this lemma plays a crucial role in Section \ref{sec_E}.

\begin{proof}
	Let $h\in H$ be arbitrary and $S$ be the collection of $\bar{q}=\langle \dot{q}_m:m<\omega\rangle$ such that for some increasing $\langle \xi_m<\gamma:m<\omega \rangle$, $\Vdash_{\p^-_{\xi_m}}\dot{q}_m\in\dot{Q}_{\xi_m,h(\xi_m)}$ holds for each $m<\omega$.
	For $\bar{q}\in S$, let $\dot{A}(\bar{q})\coloneq\{m<\omega:\dot{q}_m\in\dot{G}_\gamma\}$.
	By excluding triviality, we may assume that we are in the case $\cf(\gamma)=\omega$ (hence $\xi_m\to\gamma$) and all we have to show is the following:
	\begin{equation}
		\Vdash_{\p_\gamma}\text{``}\bigcup_{\xi<\gamma}\dot{D}^h_\xi\cup\{\dot{A}(\bar{q}):\bar{q}\in S\}\text{ has SFIP}\text{''}.
	\end{equation} 
	If not, there exist $p\in\p_\gamma$, $\xi<\gamma$, $\p_\xi$-name $\dot{A}$ of an element of $\dot{D}^h_\xi$, $\{\bar{q}^i=\langle \dot{q}^i_m:m<\omega\rangle:i<n\}\in[S]^{<\omega}$ and increasing ordinals $\langle \xi_m^i<\gamma:m<\omega \rangle$ for $i<n$ such that $\Vdash_{\p^-_{\xi^i_m}}\dot{q}_m^i\in\dot{Q}_{\xi^i_m,h(\xi^i_m)}$ holds for $m<\omega$ and $i<n$ and the following holds:
	\begin{equation}
		\label{eq_limit_SFIP_contra}
		p\Vdash_{\p_\gamma}\dot{A}\cap \bigcap_{i<n} \dot{A}(\bar{q}^i)=\emptyset.
	\end{equation}
	We may assume that $p\in\p_\xi$.
	Since all $\langle \xi^i_m<\gamma:m<\omega \rangle$ are increasing and converge to $\gamma$, there is $m_0<\omega$ such that $\xi_m^i>\xi$ for any $m>m_0$ and $i<n$.
	By Induction Hypothesis, $p\Vdash_{\p_\xi}\text{``}\dot{D}^h_\xi$ is  an ultrafilter'' and hence we can pick $q\leq_{\p_\xi} p$ and $m>m_0$ 
	such that $q\Vdash_{\p_\xi}m\in\dot{A}$.
	Let us reorder $\{\xi_m^i:i<n\}=\{\xi^0<\cdots<\xi^{l-1}\}$.
	Inducting on $j<l$, we construct $q_j\in\p_{\xi_j}$. 
	Let $q_{-1}\coloneq q$ and $j<l$ and assume we have constructed $q_{j-1}$.
	Let $I_j\coloneq\{i<n:\xi_m^i=\xi^j\}$.
	Since $\p_{\xi^j}$ forces that all $\dot{q}^i_m$ for $i\in I_j$ are in the same centered component $\dot{Q}_{\xi^j,h(\xi^j)}$, we can pick $p_j\leq q_{j-1}$ in $\p_{\xi^j}$ and 
	a $\p_{\xi^j}$-name $\dot{q}_j$ of a condition in $\qd_{\xi^j}$ such that for each $i\in I_j$, 
	$p_j\Vdash_{\p_{\xi^j}}\dot{q}_j\leq\dot{q}^i_m$.
	Let $q_j\coloneq p_j^\frown\dot{q}_j$.
	By construction, $q^\prime\coloneq q_{l-1}$ satisfies $q^\prime\leq q\leq p$ and $q^\prime\on\xi_m^i\Vdash_{\p_{\xi_m}} q^\prime(\xi_m^i)\leq \dot{q}^i_m$ for all $i<n$, so in particular,  $q^\prime\Vdash_{\p_\gamma}m\in\dot{A}\cap \bigcap_{i<n} \dot{A}(\bar{q}^i)$, which contradicts \eqref{eq_limit_SFIP_contra}.
\end{proof}

Ultrafilter-limits keep $\bb$ small and closed-ultrafilter-limits keep $\eeb$ small:

\begin{thm}(\cite[Main Lemma 4.6]{GMS16}, \cite[Lemma 1.31.]{GKS})
	\label{thm_uf_limit_keeps_b_small}
	Consider the case $\Gamma_\mathrm{uf}=\Lambda^\mathrm{lim}_\mathrm{uf}$.
	
	Assume:
	\begin{itemize}
		\item $\kappa$ is uncountable regular and $\kappa<\gamma$.
		\item $H$ is complete and has size $<\kappa$.
	\end{itemize} 
	Then, $\p_\gamma$ forces $C_{[\gamma]^{<\kappa}}\lq \mathbf{D}$, in particular, $\bb\leq\kappa$. (Note that since a set of conditions following a common guardrail is centered, $\p_\gamma$ is $\kappa$-cc and hence preserves all cardinals $\geq\kappa$.) 
\end{thm}

\begin{thm}(\cite[Main Lemma 3.2.4]{Yam24})
	\label{thm_Main_Lemma}
	Consider the case $\Gamma_\mathrm{uf}=\Lambda^\mathrm{lim}_\mathrm{cuf}$.
	
	Assume:
	\begin{itemize}
		\item $\kappa$ is uncountable regular and $\kappa<\gamma$.
		\item $H$ is complete and has size $<\kappa$.
	\end{itemize} 
	Then, $\p_\gamma$ forces $C_{[\gamma]^{<\kappa}}\lq \mathbf{BPR}$, in particular, $\ee\leq\eeb\leq\kappa$.
\end{thm}

We introduce a sufficient condition for ``$Q\subseteq\p$ is (c-)uf-lim-linked''. 
The condition is described in the context of pure combinatorial properties of posets and we do not have to think about forcings. 

\begin{dfn}
	\begin{itemize}
		\label{dfn_suff}
		\item Let $D$ be an ultrafilter on $\omega$, $Q\subseteq\p$ and  $\lim^D\colon Q^\omega\to\p$.
		For $n<\omega$, let $(\star)_n$ stand for:
		\begin{align*}
			(\star)_n:& \text{``Given }\bar{q}^j=\langle q_m^j:m<\omega\rangle\in Q^\omega\text{ for }j<n\text{ and }r\leq\textstyle{\lim^D}\bar{q}^j\text{ for all }j<n,\\
			&\text{then }\{m<\omega:r \text{ and all }q_m^j \text{ for } j<n \text{ have a common extension}\}\in D\text{''}.
		\end{align*}

		$Q\subseteq\p$ is sufficiently-$D$-lim-linked (suff-$D$-lim-linked, for short, $Q\in\Lambda^\mathrm{suff}_D(\p)$) if there exists a function $\lim^D\colon Q^\omega\to\p$ satisfying $(\star)_n$ for all $n<\omega$.
		
		Moreover, $Q\subseteq\p$ is suff-c-$D$-lim-linked ($Q\in\Lambda^\mathrm{suff}_{\mathrm{c}D}(\p)$) if $\ran(\lim^D)\subseteq Q$.
		\item $Q\subseteq\p$ is suff-(c-)uf-lim-linked if $Q$ is $Q\subseteq\p$ is suff-(c-)$D$-lim-linked for any $D$.
		\item $\Lambda^\mathrm{suff}\coloneq\bigcap_D\Lambda^\mathrm{suff}_D$ and 
		$\Lambda^\mathrm{suff}_\mathrm{c}\coloneq\bigcap_D\Lambda^\mathrm{suff}_{\mathrm{c}D}$.
	\end{itemize}
\end{dfn}


\begin{lem}
	\label{lem_suff_has_uf}
	Let $D$ be an ultrafilter on $\omega$ and $Q\subseteq\p$ is suff-$D$-lim-linked witnessed by $\lim^D$.
	Then, $Q$ is uf-lim-linked with the same witness $\lim^D$.
	In particular, $\suff\subseteq\Lambda^\mathrm{lim}_\mathrm{uf}$ and $\suff_\mathrm{c}\subseteq\Lambda^\mathrm{lim}_\mathrm{cuf}$.
\end{lem}
%

\begin{lem}(\cite{GMS16})\label{lem_Ed_is_cuf}
	Eventually different forcing $\mathbb{E}$ is
	
	$\sigma$-$\left(\Lambda(\mathrm{centered})\cap\Lambda^\mathrm{suff}_\mathrm{c}\right)$-linked and hence $\sigma$-$\left(\Lambda(\mathrm{centered})\cap\Lambda^\mathrm{lim}_\mathrm{cuf}\right)$-linked.
\end{lem}

\begin{dfn}
	
	For $g\in\left(\omega+1\setminus2\right)^\omega$,
	let $g$-prediction forcing $\pr_g$ the poset which generically adds a $g$-predictor and hence increase $\ee_g$, defined as follows:
	
	$\pr_g$ consists of tuples $(d,\pi,F)$ satisfying:
	\begin{enumerate}
		\item $d\in\sq$.
		\item $\pi=\langle\pi_n:n\in d^{-1}(\{1\})\rangle$.
		\item for each $n\in d^{-1}(\{1\})$, $\pi_n$ is a finite partial function of $\prod_{k<n}g(k)\to g(n)$.
		\item $F\in[\prod_{n<\omega}g(n)]^{<\omega}$
		\item for each $f,f^\prime\in F, f\on|d|=f^\prime\on|d|$ implies $f=f^\prime$.
	\end{enumerate}
	$(d^\prime,\pi^\prime,F^\prime)\leq(d,\pi,F)$ if:
	\begin{enumerate}[(i)]

		\item $d^\prime\supseteq d$.
		\item $\forall n\in d^{-1}(\{1\}), \pi_n^\prime\supseteq\pi_n$.
		\item $F^\prime\supseteq F$.
		\item \label{item_PR_order_long}
		$\forall n\in(d^\prime)^{-1}(\{1\})\setminus d^{-1}(\{1\}),\forall f\in F, f\on n\in\dom(\pi^\prime_n)$ and $\pi^\prime_n(f\on n)=f(n)$.
		
	\end{enumerate}
	When $g(n)=\omega$ for all $n<\omega$, we write $\pr$ instead of $\pr_g$ and just call it ``prediction forcing''.
\end{dfn}

\begin{lem}
	(\cite{BS_E_and_P_2})\label{lem_PR_uf}
	For $g\in\left(\omega+1\setminus2\right)^\omega$, $\pr_g$ is $\sigma$-$\left(\Lambda(\mathrm{centered})\cap\Lambda^\mathrm{suff}\right)$-linked and hence $\sigma$-$\left(\Lambda(\mathrm{centered})\cap\Lambda_\mathrm{uf}^\mathrm{lim}\right)$-linked. Moreover, $\pr_g$ is $\sigma$-$\left(\Lambda(\mathrm{centered})\cap\Lambda_\mathrm{c}^\mathrm{suff}\right)$-linked and hence
	$\sigma$-$\left(\Lambda(\mathrm{centered})\cap\Lambda_\mathrm{cuf}^\mathrm{lim}\right)$-linked whenever $g\in(\omega\setminus2)^\omega$.
\end{lem}

	\subsection{FAM-limit}
	\newcommand{\ind}{i}
\newcommand{\pst}{\mathbb{P}}
We describe the general theory of FAM-limits following \cite{Uri}.
Also, the original ideas are already in \cite{She00} and \cite{KST}. Thus, we omit proofs in this subsection (See \cite{Uri} for details). 
\begin{dfn}
	\begin{itemize}
		\item For a set $A$ and a finite non-empty set $B$,
		\[\dns_B(A)\coloneq\frac{|A\cap B|}{|B|}.\]
		\item For a poset $\p$ and a countable sequence $\bar{p}=\langle p_l:l<\omega\rangle\in\p^\omega$ of conditions, let $\Vdash_\p\dot{W}(\bar{p})\coloneq\{l<\omega:p_l\in\dot{G}\}$.
		\item A fam on $\omega$ is a finitely additive measure $\Xi\colon\mathcal{P}(\omega)\to[0,1]$ such that $\Xi(\omega)=1$ and $\Xi(\{n\})=0$ for all $n<\omega$.
		\item $\mathbb{I}_\infty$ is defined by $\langle I_k\rangle_{k<\omega}\in\mathbb{I}_\infty:\Leftrightarrow\langle I_k\rangle_{k<\omega}$ is an interval partition of $\omega$ such that all $I_k$ are non-empty and $|I_k|$ goes to infinity.
		\item $(0,1)_\mathbb{Q}$ denotes the set of all rational numbers $q$ with $0<q<1$.
	\end{itemize}
\end{dfn}

\begin{dfn}
	Let $\Xi$ be a fam, $\bari\in\mathbb{I}_\infty$, $0<\varepsilon<1$ and $\p$ be a poset.
	A set $Q\subseteq\p$ is $(\Xi,\bar{I},\varepsilon)$-linked if there is some map $\lim^\Xi:Q^\omega\to \p$ and some $\p$-name $\dot{\Xi}^\prime$ of a fam extending $\Xi$ such that for any $\bar{q}\in Q^\omega$,
	\begin{equation*}
		\textstyle{\lim^\Xi}\displaystyle{\bar{q}\Vdash\int_\omega \dns_{I_k}(\dot{W}(\bar{q}))~d\dot{\Xi}^\prime\geq1-\varepsilon}.
	\end{equation*} 
\end{dfn}

\begin{dfn}
	Let $\p$ be a poset, $n<\omega$ and $\bar{q}\in\p^n$.
	We define \[i^\p_*(\bar{q})\coloneq\max\{|F|:F\subseteq n, \{q_i:i\in F\}\text{ has a common lower bound}\}.\]
	For $Q\subseteq\p$, define
	\[\intt(Q)\coloneq\inf_{0<n<\omega}\left\{\frac{i^\p_*(\bar{q})}{n}:\bar{q}\in Q^n\right\}.\]
\end{dfn}

\begin{dfn}
	A poset $\p$ is $\mu$-FAM-linked\footnote{This linkedness notion is not in the framework of Definition \ref{dfn_linked_class} in that we use an additional parameter $\varepsilon\in(0,1)_\mathbb{Q}$, but the formalization is almost the same as ultrafilter-limits.} if there is a sequence $\langle Q_{\a,\varepsilon} :\a<\mu, \varepsilon\in(0,1)_\mathbb{Q}\rangle$ of subsets of $\p$ such that\footnote{Our definition is a bit different from that in \cite{Uri}: In Definition 4.2.8 in \cite{Uri} he required a stronger property for $\mu$-FAM-linkedness, especially for $(\Xi,\bar{I},\varepsilon)$-linkedness in Definition 4.2.2, and dealt with our definition as a derived property in Theorem 4.2.5.}:
	\begin{itemize}
		\item Each $Q_{\a,\varepsilon}$ is $(\Xi,\bar{I},\varepsilon)$-linked for any fam $\Xi$ and $\bar{I}\in \mathbb{I}_\infty$.
		\item For any $\varepsilon\in(0,1)_\mathbb{Q}$, $\bigcup_{\a<\mu} Q_{\a,\varepsilon}$ is dense in $\p$.
		\item For any $\a<\mu$ and $\varepsilon\in(0,1)_\mathbb{Q}$, $\intt(Q_{\a,\varepsilon})\geq1-\varepsilon$.
	\end{itemize}
\end{dfn}

\begin{dfn}
	\label{dfn_FAM_iteration}
	\begin{itemize}
		\item A $\kappa$-FAM-iteration is a fsi =$\langle\p_\xi, \qd_\xi : \xi<\gamma\rangle$ with witnesses $\langle\p^-_\xi:\xi<\gamma\rangle$ and $\bar{Q}=\langle\dot{Q}^\xi_{\zeta,\varepsilon} : \xi<\gamma,\zeta<\theta_\xi, \varepsilon\in(0,1)_\mathbb{Q}\rangle$ satisfying for all $\xi<\gamma$:
		\begin{enumerate}
			\item $\p^-_\xi\lessdot\p_\xi$.
			\item $\theta_\xi<\kappa$.
			\item $\p^-_\xi$ forces that $\qd_\xi $ is $\theta_\xi$-FAM-linked witnessed by
			
			$\bar{Q}^\xi\coloneq\langle\dot{Q}^\xi_{\zeta,\varepsilon} : \zeta<\theta_\xi, \varepsilon\in(0,1)_\mathbb{Q}\rangle$.
		\end{enumerate}
		\item  $\xi<\gamma$ is a trivial stage if $\Vdash_{\p^-_\xi}|\dot{Q}^\xi_{\zeta,\varepsilon}|=1$ for all $\zeta<\theta_\xi$ and $\varepsilon\in(0,1)_\mathbb{Q}$. $S^-$ is the set of all trivial stages and $S^+\coloneq\gamma\setminus S^-$.
		\item A guardrail is a function $h\in \prod_{\xi<\gamma}(\theta_\xi\times(0,1)_\mathbb{Q})$.
		For a guardrail $h$, we write $h(\xi)=(h_L(\xi),h_R(\xi))$.
		\item $\p^h_\eta$ is the set of conditions $p\in\p_\eta$ following $h$, i.e., for $\xi\in\dom(p)$,
		$p(\xi)$ is a $\p^-_\xi$-name and $\Vdash_{\p^-_\xi} p(\xi)\in \dot{Q}^\xi_{h(\xi)}$.
		\item $\bar{q}=\langle \dot{q}_m:m<\omega\rangle$  sequentially follows $h$ with constant $\varepsilon\in(0,1)_\mathbb{Q}$ if there are $\langle\xi_m:m<\omega\rangle\in\gamma^\omega$ such that both $h_R(\xi_m)=\varepsilon$ and 
		$\Vdash_{\p^-_{\xi_m}}\dot{q}_m\in \dot{Q}^{\xi_m}_{h(\xi_m)}$ hold for all $m<\omega$.

		\item A set of guardrail $H$ is complete if for any guardrail $h$ and any countable partial function $h_0\subseteq h$, there exists $h^\prime\in H$ extending $h_0$.
	\end{itemize}
	
\end{dfn}

\begin{dfn}
	A $\kappa$-FAM-iteration has FAM-limits on $H$ with witness $\langle\Xi_\xi^{h,\bar{I}}:h\in H, \bar{I}\in\mathbb{I}_\infty,\xi\leq\gamma\rangle$ if
	\begin{enumerate}
		\item $H\subseteq\prod_{\xi<\gamma}(\theta_\xi\times(0,1)_\mathbb{Q})$, a set of guardrails.
		\item For $h\in H$, $\bar{I}\in\mathbb{I}_\infty$ and $\xi\leq\gamma$, $\dot{\Xi}_\xi^{h,\bar{I}}$ is a $\p_\xi$-name of a fam.
		\item If $\xi<\eta\leq\gamma$, then $\Vdash_{\p_\eta}\dot{\Xi}_\xi^{h,\bar{I}}\subseteq\dot{\Xi}_\eta^{h,\bar{I}}$.
		\item $\Vdash_{\p_\xi}(\dot{\Xi}_\xi^{h,\bar{I}})^-\in V^{\p^-_\xi}$, where $(\dot{\Xi}_\xi^{h,\bar{I}})^-\coloneq\dot{\Xi}_\xi^{h,\bar{I}}\cap V^{\p^-_\xi}$ if $\xi\in S^+$, otherwise let $(\dot{\Xi}_\xi^{h,\bar{I}})^-$ be an arbitrary fam in $V^{\p^-_\xi}$ (hence this item is trivially satisfied in this case).
		\item Whenever $\bar{q}=\langle \dot{q}_m:m<\omega\rangle$ sequentially follows $h$ with constant $\varepsilon\in(0,1)_\mathbb{Q}$ with witnesses $\langle\xi_m:m<\omega\rangle\in\gamma^\omega$, 
		\begin{enumerate}
			\item If $\langle\xi_m:m<\omega\rangle$ is constant with value $\xi$, then
			\begin{equation}
				\Vdash_{\p_\xi}\textstyle{\lim^{\dot{\Xi}_\xi^{h,\bar{I}}}\bar{q}}~\displaystyle{\Vdash_{\qd_\xi}\int\dns_{I_k}(\dot{W}(\bar{q}))d\dot{\Xi}_{\xi+1}^{h,\bar{I}}\geq1-\varepsilon}.
			\end{equation}
			\item If $\langle\xi_m:m<\omega\rangle$ is increasing, then for any $\varepsilon^\prime>0$,
			\begin{equation}
				\Vdash_{\p_\gamma}\dot{\Xi}_\gamma^{h,\bar{I}}\left(\{k<\omega:\dns_{I_k}(W(\bar{q}))\geq(1-\varepsilon)\cdot(1-\varepsilon^\prime)\}\right)=1.
			\end{equation}
		\end{enumerate}
		
	\end{enumerate}
\end{dfn}

\begin{lem}(\cite[Theorem 4.3.16, 4.3.18.]{Uri})
	\label{lem_fam_const_all}
	Let $\gamma$ be an ordinal and 
	$\p_\gamma$ be a $\kappa$-FAM-iteration.
	If $\langle \dot{\Xi}^{h,\bar{I}}_\xi:\xi<\gamma,~h\in H, \bar{I}\in\mathbb{I}_\infty\rangle$ witnesses that for any $\xi<\gamma$, $\p_\xi=\p_\gamma\on\xi$ has FAM-limits on $H$,  
	then we can find $\langle \dot{\Xi}^{h,\bar{I}}_\gamma:h\in H, \bar{I}\in\mathbb{I}_\infty\rangle$ such that $\langle \dot{\Xi}^{h,\bar{I}}_\xi:\xi\leq\gamma,~h\in H, \bar{I}\in\mathbb{I}_\infty\rangle$  witnesses $\p_\gamma$ has FAM-limits on $H$.
\end{lem}

\begin{thm}(\cite{Car23RIMS}, detailed description of Theorem \ref{thm_FAM_nonE}.)
	\label{thm_FAM_keeps_nonE}
	Let $H$ be a set of guardrails for a $\kappa$-FAM-iteration and $\p_\gamma$ be a $\kappa$-FAM-iteration with FAM-limits on $H$.
	Assume:
	\begin{itemize}
		\item $\kappa$ is uncountable regular and $\kappa<\gamma$.
		\item $H$ is complete and has size $<\kappa$.
	\end{itemize} 
	Then, $\p_\gamma$ forces\footnote{Similarly to the case of ultrafilter-limits, $\p_\gamma$ is $\kappa$-cc and hence preserves all cardinals $\geq\kappa$ (see e.g. \cite{Uri}).} $C_{[\gamma]^{<\kappa}}\lq C_\mathcal{E}$, in particular, $\non(\mathcal{E})\leq\kappa$.
\end{thm}
\subsection{The forcing poset $\widetilde{\mathbb{E}}$}

In \cite{KST}, they also introduced the forcing notion  $\widetilde{\mathbb{E}}$ instead of $\mathbb{E}$, which increases $\nonm$ as $\mathbb{E}$ does, but has FAM-limits,  which $\mathbb{E}$ does not have. (Cardona, Mej\'{\i}a and Uribe-Zapata \cite{Car23RIMS} proved that $\mathbb{E}$ increases $\non(\mathcal{E})$ and hence does not have FAM-limits.)
We give a brief explanation of $\widetilde{\mathbb{E}}$.
(See \cite{KST} for details) 
\begin{dfn}
	Let $T\subseteq\seq$ be a tree.
	\begin{enumerate}
		\item $\stem(T)$ denotes the stem of $T$, i.e., the smallest splitting node.
		\item For $s,t\in T$, $s\triangleleft t$ denotes that $t$ is an immediate successor, i.e., $t=s^\frown l$ for some $l<\omega$. Define $\Omega_s=\Omega_s(T)\coloneq\{s^\prime\in T:s^\prime\triangleright s\}$, the set of all immediate successors of $s$ in $T$. 
	\end{enumerate}
\end{dfn}

\begin{dfn}
	\begin{enumerate}
		\item We define the compact homogeneous tree $T^*\subseteq\seq$ with the empty stem $\stem(T^*)=\langle\rangle$ by induction on the height $h$.
		For $h<\omega$, let $\rho(h)\coloneq\max\{|T^*\cap\omega^h|,h+2\}$, $\pi(h)\coloneq((h+1)^2\rho(h)^{h+1})^{\rho(h)^h}$, $a(h)\coloneq\pi(h)^{h+2}$, $M(h)\coloneq a(h)^2$ and $\mu_h(n)\coloneq\log_{a(h)}\left(\displaystyle{\frac{M(h)}{M(h)-n}}\right)$ for $n\leq M(h)$ ($\mu_h(M(h))=\infty)$.
		
		For $s\in T^*\cap\omega^h$, let $\Omega_s=\Omega_s(T^*)=\{s^\frown l:l<M(h)\}$ and now $T^*$ is inductively defined. 
		\item For $s\in T^*$ and $A\subseteq \Omega_s$, define $\mu_s(A)\coloneq\mu_{|s|}(|A|)$. For a subtree $p\subseteq T^*$,
		define $\mu_s(p)=\mu_s(\{s^\prime\in p:s^\prime\triangleright s\})$ for $s\in p$.
		\item 
		The poset $\widetilde{\mathbb{E}}$ consists of all subtrees $p\subseteq T^*$ such that $\mu_s(p)\geq1+\frac{1}{|\stem(p)|}$ for all $s\in p$ with $s\supseteq\stem(p)$. $p\leq q:\Leftrightarrow p\subseteq q$.
		
	\end{enumerate}
\end{dfn}

\begin{fac}
	\label{fac_E}
	\begin{enumerate}
		\item\label{item_fac_E_a} (\cite[Lemma 1.19.(a)]{KST})  Let $h<\omega$, $s\in\omega^h$ and $\{p_i:i<\pi(h)\}\subseteq\widetilde{\mathbb{E}}$.
		If all $p_i$ share $s$ above their stems, then they have a common lower bound.

		\item \label{item_fac_E_b}(\cite[Lemma 1.19.(b)]{KST}) $\widetilde{\mathbb{E}}$ is $(\rho,\pi)$-linked,
		
		i.e.,  $\widetilde{\mathbb{E}}=\bigcup_{m<\omega}\bigcap_{i>m}\bigcup_{j<\rho(i)}Q^i_j$ where each $Q^i_j$ is $\pi(i)$-linked.

		\item \label{item_fac_E_c}(\cite[Lemma 1.19.(c)]{KST}) $\widetilde{\mathbb{E}}$ adds an eventually different real and hence increases $\nonm$.
		\item\label{item_fac_E_e} (\cite[Lemma 1.19.(e)]{KST}) $\widetilde{\mathbb{E}}$ is (forcing equivalent to) a subforcing of random forcing.
	\end{enumerate}
\end{fac}


\begin{lem}(\cite[Lemma 1.20]{KST},\cite[Theorem 4.2.19]{Uri})\label{lem_E_FAM}
	
	For $t\in T^*$ and $\varepsilon\in(0,1)_\mathbb{Q}$, let $Q_{t,\varepsilon}$ denote the set of all $p\in\widetilde{\mathbb{E}}$ such that some natural number $m\geq2$ satisfies:
	\begin{itemize}
		\item $\stem(p)=t$ has length $>3m$.
		\item $\mu_s(p)\geq1+\frac{1}{m}$ for all $s\in p$ with $s\supseteq t$.
		\item $\frac{1}{m}\leq\varepsilon$.
	\end{itemize} 
	Then, $\langle Q_{t,\varepsilon}:t\in T^*,\varepsilon\in(0,1)_\mathbb{Q}\rangle$ witnesses
	$\widetilde{\mathbb{E}}$ is $\sigma$-FAM-linked.
\end{lem}

Actually, $\widetilde{\mathbb{E}}$ has also closed ultrafilter-limits and this is why we say we are close to achieving an iteration with all the three limit notions. 
\begin{lem}(\cite{GKMS_CM_and_evasion})
	\label{lem_E_cuf}
	$Q=Q_{t,\varepsilon}\subseteq\widetilde{\mathbb{E}}$ is closed-ultrafilter-limit-linked for any $t\in T^*$ and $\varepsilon\in(0,1)_\mathbb{Q}$.
	Thus, $\widetilde{\mathbb{E}}$ is $\sigma$-$\Lambda_\mathrm{cuf}^\mathrm{lim}$-linked.
\end{lem}

\begin{proof}
	Let $D$ be any ultrafilter on $\omega$.
	We define $\lim^D\colon Q^\omega\to Q$ as follows:
	For $\bar{q}=\langle q_m:m<\omega\rangle\in Q^\omega$, define $q^\infty=\lim^D\bar{q}\coloneq\{t\in T^*:X_t\in D\}$ where $X_t\coloneq\{m<\omega:t\in q_m\}$.
	Since $s\subseteq t$ implies $X_s\supseteq X_t$,
	$q^\infty$ is a subtree of $T^*$.
	Inducting on $h<\omega$, we show:
	\begin{equation}
		\label{eq_Etilde_lim_lem}
		Y_h\coloneq\{m<\omega:q^m\cap\omega^h=q^\infty\cap\omega^h\}\in D \text{ for any }h<\omega.
	\end{equation}
	$Y_0=\omega\in D$ trivially holds. Let $h<\omega$ and assume $Y_h\in D$.
	Fix $s\in q^\infty\cap\omega^h$.
	For $A\subseteq\Omega_s$,
	define $Z_A\coloneq\{m<\omega:q_m\cap\Omega_s=A\}$.
	Since $\bigcup\{Z_A:A\subseteq\Omega_s\}$ is a finite partition of $\omega$, there uniquely exists $A\subseteq\Omega_s$ such that $Z_A\in D$.
	We show $A=q^\infty\cap\Omega_s$.
	On the one hand, if $t\in A$, then $t\in q_m$ holds for any $m\in Z_A$,
	so $Z_A\subseteq X_t$ and hence $t\in q^\infty$.
	On the other hand, if $t\in q^\infty\cap\Omega_s$,
	then $X_t\in D$ and we can pick some $m\in X_t\cap Z_A$.
	$q_m$ witnesses $t\in A$.
	Thus, $A=q^\infty\cap\Omega_s$ holds.
	Let $W_s$ be the uniquely determined set $Z_A$ for $s\in q^\infty\cap\omega^h$.
	Then, $Y_h\cap\bigcap\{W_s:s\in q^\infty\cap\omega^h\}\in D$ is a subset of $Y_{h+1}$ and hence the induction is done.
	By \eqref{eq_Etilde_lim_lem}, we obtain $\stem(q^\infty)=t_0$,
	$q^\infty\in\widetilde{\mathbb{E}}$ and moreover $q^\infty\in Q$ (hence the closedness of $\lim^D$ is shown).
	To show $(\star)_n$ in Definition \ref{dfn_suff} for $n<\omega$, let $\bar{q}^j=\langle q_m^j:m<\omega\rangle\in Q^\omega$ for $j<n$ and $r\leq\lim^D\bar{q}^j$ for all $j<n$.
	Take $h<\omega$ satisfying $h>n,|\stem(r)|$ and $t\in r\cap\omega^h$.
	Since $t$ is also in each $\lim^D\bar{q}^j$, there exists $X\in D$ such that for all $m\in X$, $t\in q_m^j$ for all $j<n$.
	Thus, $\leq h(\leq\pi(h))$-many conditions $r$ and $\{q_m^j:j<n\}$ share the common node $t$ above their stems, by Fact \ref{fac_E} \eqref{item_fac_E_a} they have a common extension.
	
	Therefore, 	$\lim^D$ witnesses $Q\subseteq\widetilde{\mathbb{E}}$ is closed-$D$-limit-linked.

\end{proof}

	\section{Iteration}\label{sec_E}
	
Table \ref{table} illustrates the relationship among the limit notions, the forcing notions and the cardinal invariants.
The reason why each line in the table holds is as follows:
\begin{enumerate}
	\item By Lemma \ref{lem_PR_uf} and the fact that $\pr$ increases $\ee$, which is kept small by closed-ultrafilter-limits and FAM-limits.
	\item By Lemma \ref{lem_PR_uf} and the fact that $\pr_g$ increases $\ee_g$, which is kept small by FAM-limits.
	\item By Lemma \ref{lem_E_cuf} and \ref{lem_E_FAM}.
	\item By Theorem \ref{thm_uf_limit_keeps_b_small}, \ref{thm_Main_Lemma} and \ref{thm_FAM_keeps_nonE}.
\end{enumerate}
Thus, we naturally expect that by iterating (subforcings of) each poset we can obtain the left side of the separation constellation illustrated in Figure \ref{fig_p7}, since the smallness of $\bb$, $\eeb$ and $\non(\mathcal{E})$ is guaranteed by each limit notion.

Let us get ready for the iteration. The construction follows \cite[Section 4]{Yam24}.

\begin{table}[b]
	\caption{The kinds of the limits each forcing notion has.}\label{table} 
	\centering
	\begin{tabular}{ccccc}
		\hline
		forcing notion&  ultrafilter  & closed-ultrafilter & FAM \\
		
		\hline\hline
		
		$\mathbb{PR}$   & $	\textcolor{black}{\checkmark}$  & $\times$& $\times$ \\
		$\mathbb{PR}_g$    & $	\textcolor{black}{\checkmark}$& $	\textcolor{black}{\checkmark}$& $\times$ \\
		$\widetilde{\mathbb{E}}$    &  $	\textcolor{black}{\checkmark}$ & $	\textcolor{black}{\checkmark}$& $	\textcolor{black}{\checkmark}$\\
		
		\hline 
		keep small & $\bb$ & $\eeb$&$\non(\mathcal{E})$\\
		\hline
	\end{tabular}
	
\end{table}

\begin{dfn}
	\begin{itemize}
		\item $\R_1\coloneq\mathbf{Lc}^*$ and $\br_1\coloneq\mathbb{A}$.
		\item $\R_2\coloneq\mathbf{Cn}$ and $\br_2\coloneq\mathbb{B}$.
		\item $\R_3\coloneq\mathbf{D}$ and $\br_3\coloneq\mathbb{D}$.
		\item $\R_4\coloneq\mathbf{PR}$, $\R_4^*\coloneq\mathbf{BPR}$ and $\br_4\coloneq\mathbb{PR}$.
		\item $\R_5\coloneq C_\mathcal{E}$, $\R_5^g\coloneq \mathbf{PR}_g$ , and $\br_5^g\coloneq\pr_g$ for $g\in(\omega\setminus2)^\omega$.
		\item $\R_6\coloneq\mathbf{Mg}$ and $\br_6\coloneq\widetilde{\mathbb{E}}$.
	\end{itemize}
	Let $I\coloneq\{1,\ldots,6\}$ be the index set.
\end{dfn}

$\br_{\ind}$ is the poset which increases $\bb(\R_\ind)$ for each $\ind\in I$.
\begin{ass}
	\label{ass_card_arith_E}
	
	\begin{enumerate}
		\item $\lambda_1<\cdots<\lambda_7$ are regular uncountable cardinals.
		\item $\lambda_3=\mu_3^+$, $\lambda_4=\mu_4^+$ and $\lambda_5=\mu_5^+$ are successor cardinals and $\mu_3$ is regular.
		\item \label{item_aleph1_inacc}$\kappa<\lambda_\ind$ implies $\kappa^{\aleph_0}<\lambda_\ind$ for all $\ind\in I$.
		\item \label{item_ca_3}
		$\lambda_7^{<\lambda_6}=\lambda_7$, hence $\lambda_7^{<\lambda_\ind}=\lambda_7$ for all $\ind\in I$.

	\end{enumerate}
\end{ass}

\begin{dfn}
	Put $\gamma\coloneq\lambda_7$, the length of the iteration we shall perform. Fix $S_1\cup\cdots\cup S_6=\gamma$, a cofinal partition of $\lambda_7$ and for $\xi<\gamma$, let $\ind(\xi)$ denote the unique $\ind\in I$ such that $\xi\in S_\ind$.
\end{dfn}

\begin{ass}
	\label{ass_for_complete_guardrail_E}
	$\lambda_7\leq2^{\mu_3}$.
\end{ass}

\begin{lem}
	There exist complete sets $H$, $H^\prime$ and $H^{\prime\prime}$ of guardrails of length $\gamma=\lambda_7$ for $\lambda_3$-$\Lambda^\mathrm{lim}_\mathrm{uf}$-iteration of size $<\lambda_3$, $\lambda_4$-$\Lambda^\mathrm{lim}_\mathrm{cuf}$-iteration of size $<\lambda_4$ and $\lambda_5$-FAM-iteration of size $<\lambda_5$, respectively.  
\end{lem}

Now we are ready to describe the iteration construction, which is the main object to analyze in this article.
\begin{con}
	\label{con_P7}
	We shall construct a ccc finite support iteration $\pst^7_\mathrm{pre}$ satisfying the following items:
	\begin{enumerate}
		\item \label{item_P7_1}
		$\pst^7_\mathrm{pre}$ is a $\lambda_3$-$\Lambda^\mathrm{lim}_\mathrm{uf}$-iteration of length $\gamma\coloneq\lambda_7$ and has $\Lambda^\mathrm{lim}_\mathrm{uf}$-limits on $H$ 
		with the following witnesses:
		\begin{itemize}
			\item $\langle\p_\xi^-:\xi<\gamma\rangle$, the complete subposets witnessing $\Lambda^\mathrm{lim}_\mathrm{uf}$-linkedness.
			\item $\bar{Q}=\langle\dot{Q}_{\xi,\zeta}:\zeta<\theta_\xi,\xi<\gamma\rangle$, the $\Lambda^\mathrm{lim}_\mathrm{uf}$-linked components.
			\item $\bar{D}=\langle \dot{D}^h_\xi:\xi\leq\gamma,~h\in H \rangle$, the ultrafilters.
			\item $S^-\coloneq S_1\cup S_2\cup S_3$, the trivial stages and $S^+\coloneq S_4\cup S_5\cup S_6$, the non trivial stages.
		\end{itemize}
		\item \label{item_P7_2}
		$\pst^7_\mathrm{pre}$ is a $\lambda_4$-$\Lambda^\mathrm{lim}_\mathrm{cuf}$-iteration and has $\Lambda^\mathrm{lim}_\mathrm{cuf}$-limits on $H^\prime$ 
		with the following witnesses:
		\begin{itemize}
			\item $\langle\p_\xi^-:\xi<\gamma\rangle$, the complete subposets witnessing $\Lambda^\mathrm{lim}_\mathrm{cuf}$-linkedness.
			\item $\bar{R}=\langle\dot{R}_{\xi,\zeta}:\zeta<\theta_\xi,\xi<\gamma\rangle$, the $\Lambda^\mathrm{lim}_\mathrm{cuf}$-linked components.
			\item $\bar{E}=\langle \dot{E}^{h^\prime}_\xi:\xi\leq\gamma,~h^\prime\in H^\prime\rangle$, the ultrafilters.
			\item $T^-\coloneq S_1\cup S_2\cup S_3\cup S_4$, the trivial stages and $T^+\coloneq S_5\cup S_6$, the non-trivial stages.
		\end{itemize}
		\item \label{item_P7_3}
		$\pst^7_\mathrm{pre}$ is a $\lambda_5$-FAM-iteration and has FAM-limits on $H^{\prime\prime}$
		with the following witnesses:
		\begin{itemize}
			\item $\langle\p_\xi^-:\xi<\gamma\rangle$, the complete subposets witnessing FAM-linkedness.
			\item $\bar{S}=\langle\dot{S}^\varepsilon_{\xi,\zeta}:\varepsilon\in(0,1)_\mathbb{Q},\zeta<\theta_\xi,\xi<\gamma\rangle$, the FAM-linked components.
			\item $\bar{\Xi}=\langle \dot{\Xi}^{h^{\prime\prime},\bar{I}}_\xi:\xi\leq\gamma,~h^{\prime\prime}\in H^{\prime\prime}, \bar{I}\in\mathbb{I}_\infty\rangle$, the finitely additive measures.
			\item $U^-\coloneq S_1\cup S_2\cup S_3\cup S_4\cup S_5$, the trivial stages and $U^+\coloneq S_6$, the non-trivial stages.
		\end{itemize}
		\item 
		\label{item_N_7}
		For each $\xi<\gamma$, $N_\xi\preccurlyeq H_\Theta$ is a submodel where $\Theta$ is a sufficiently large regular cardinal satisfying:
		\begin{enumerate}
			\item $|N_\xi|<\lambda_{\ind(\xi)}$.
			\item \label{item_sigma_closed_E}
			$N_\xi$ is $\sigma$-closed, i.e., $(N_\xi)^\omega\subseteq N_\xi$.
			
			\item 
			\label{item_N_e_E}For any $\ind\in I$, $\eta<\gamma$ and set of (nice names of) reals $A$ in $V^{\p_\eta}$ of size $<\lambda_\ind$,
			there is some $\xi\in S_\ind$ (above $\eta$) such that $A\subseteq N_\xi$.
			\item If $\ind(\xi)=4$, then $\{\dot{D}^h_\xi:h\in H \}\subseteq N_\xi$.
			\item If $\ind(\xi)=5$, then $\{\dot{D}^h_\xi:h\in H \},\{\dot{E}^{h^\prime}_\xi:h^\prime\in H^\prime \}\subseteq N_\xi$.
			\item If $\ind(\xi)=6$, then $\{\dot{D}^h_\xi:h\in H \},\{\dot{E}^{h^\prime}_\xi:h^\prime\in H^\prime \}, \{ \dot{\Xi}^{h^{\prime\prime},\bar{I}}_\xi:h^{\prime\prime}\in H^{\prime\prime}, \bar{I}\in\mathbb{I}_\infty\}\subseteq N_\xi$.
			
		\end{enumerate}
		
		\item $\p_\xi^-\coloneq\p_\xi\cap N_\xi$. 
		
		\item For each $\xi<\gamma$,
		$\p^-_\xi\Vdash\qd_\xi\coloneq\br_{\ind(\xi)}$. 
		
		($\br_5$ denotes $\br_5^g$ for some $g\in(\omega\setminus2)^\omega$ and $g$ runs through all $g\in(\omega\setminus 2)^\omega$ by bookkeeping.)
		\item $\bar{Q},\bar{R}$ and $\bar{S}$ are determined in the canonical way: at trivial stages, they consist of singletons and at non-trivial stages, they consist of $\omega$-many $\Lambda^\mathrm{lim}_\mathrm{uf}/\Lambda^\mathrm{lim}_\mathrm{cuf}/$FAM-linked components, respectively.


	\end{enumerate}

\end{con}
To begin with the conclusion, \textit{we have no idea whether Construction \ref{con_P7} is really possible}.
We first describe what separation constellation we can obtain \textit{if the construction is possible}.
Then, we analyze the construction and summarize what we can do and what we may not be able to do.
\begin{thm}
	If Construction \ref{con_P7} is possible, then $\pst^7_\mathrm{pre}$ will force for each $\ind\in I$,
	$\R_\ind\cong_T C_{[\lambda_7]^{<\lambda_\ind}}\cong_T[\lambda_7]^{<\lambda_\ind}$, in particular, $\bb(\R_\ind)=\lambda_\ind$ and $\dd(\R_\ind)=2^{\aleph_0}=\lambda_7$ (the same things also hold for $\R_4^*$ and $\R_5^g$ for $g\in(\omega\setminus2)^\omega)$.
	Moreover, by recovering GCH and applying the submodel method, we can obtain $\pst^7_\mathrm{fin}$ which gives the constellation illustrated in Figure \ref{fig_p7}, where
	$\aleph_1\leq\theta_1\leq\cdots\leq\theta_{12}$ are regular and $\theta_\cc$ is an infinite cardinal such that $\theta_\cc\geq\theta_{12}$ and $\theta_\cc^{\aleph_0}=\theta_\cc$.
\end{thm}

\begin{proof}
	Compared with $\p_\mathrm{pre}$ in \cite[Construction 4.7]{Yam24}, what we have to additionally deal with are the following:
	\begin{enumerate}
		\item Why $C_{[\lambda_7]^{<\lambda_1}}\lq\R_1$ holds: By Fact \ref{fac_E}\eqref{item_fac_E_e} and Corollary \ref{cor_smallness_for_addn_and_covn_and_nonm}.
		\item Why $C_{[\lambda_7]^{<\lambda_2}}\lq\R_2$ holds: It is known that there exists a Prs $\R_2^\prime$ such that $\R_2^\prime\lq\R_2$ and $(\rho,\pi)$-linked forcings are $\R^\prime_2$-good (\cite[Definition 2.3., Lemma 2.5.3., Lemma 2.8.3.]{KST}). Use Fact \ref{fac_E} \eqref{item_fac_E_b}. 
		\item Why $C_{[\lambda_7]^{<\lambda_5}}\lq\R_5$ holds: By Theorem \ref{thm_FAM_keeps_nonE}.
		\item Why $\R_6\lq C_{[\lambda_7]^{<\lambda_6}}$ holds: By Fact \ref{fac_E} \eqref{item_fac_E_c}. 
	\end{enumerate}
	
\end{proof}

Let us analyze Construction \ref{con_P7}.
Actually we can realize the successor steps.

\textit{Successor step.} If $\ind(\xi)\leq5$, we are in a similar case to \cite[Construction 4.7]{Yam24} since $\xi$ is a trivial stage for the FAM-iteration and we do not have to think about fams.
Thus, we may assume that $\ind(\xi)=6$.
$\br_6=\widetilde{\mathbb{E}}$ is $\sigma$-$\Lambda_\mathrm{cuf}^\mathrm{lim}$-linked by Lemma \ref{lem_E_cuf}, so we can extend the ultrafilters $\bar{D}$ and $\bar{E}$.
By Lemma \ref{lem_E_FAM} and Lemma \ref{lem_fam_const_all}, we can also extend the fams $\bar{\Xi}$.
Also, we can take a submodel $N_\xi$ since $|H^{\prime\prime}\times\mathbb{I}_\infty|<\lambda_6$.

However, we have a problem at limit steps, not on fams, but on ultrafilters.

\textit{Limit step.}
For fams $\bar{\Xi}$, we can directly apply Lemma \ref{lem_fam_const_all}.
For ultrafilters $\bar{D}$ and $\bar{E}$, recall that when constructing the names of ultrafilters at limit steps in the proof of Lemma \ref{lem_uf_const_all}, we resorted to the \textit{centeredness}, which $\widetilde{\mathbb{E}}$ \textit{does not have}.
We have no idea on how to overcome this problem without resorting to centeredness.\footnote{Goldstern, Kellner, Mej\'{\i}a and Shelah \cite{GKMS_CM_and_evasion} first stated that they added the evasion number $\ee$ to Cicho\'n's maximum by performing an iteration with both ultrafilter-limits and FAM-limits, but later found a gap in their proof, precisely in this point.}.

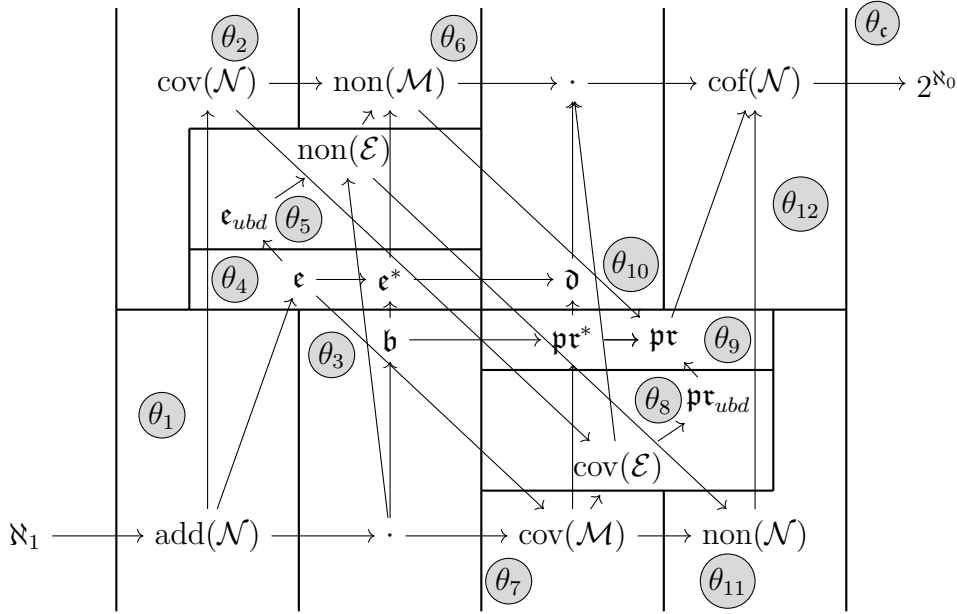
\begin{figure}
	\centering
	\begin{tikzpicture}
		\tikzset{
			textnode/.style={text=black}, 
		}
		\tikzset{
			edge/.style={color=black, thin, opacity=0.4}, 
		}
		\newcommand{\w}{2.4}
		\newcommand{\h}{2.0}
		
		\node[textnode] (addN) at (0,  0) {$\addn$};
		\node (t1) [fill=gray!30, draw, text=black, circle,inner sep=1.0pt] at (-0.25*\w, 0.8*\h) {$\theta_1$};
		
		\node[textnode] (covN) at (0,  \h*3) {$\covn$};
		\node (t2) [fill=gray!30, draw, text=black, circle,inner sep=1.0pt] at (0.15*\w, 3.3*\h) {$\theta_2$};

		\node[textnode] (addM) at (\w,  0) {$\cdot$};
		\node[textnode] (b) at (\w,  1.3*\h) {$\bb$};
		\node (t3) [fill=gray!30, draw, text=black, circle,inner sep=1.0pt] at (0.68*\w, 1.2*\h) {$\theta_3$};
		
		\node[textnode] (nonM) at (\w,  \h*3) {$\nonm$};
		\node (t6) [fill=gray!30, draw, text=black, circle,inner sep=1.0pt] at (1.35*\w, 3.3*\h) {$\theta_6$};
		
		\node[textnode] (covM) at (\w*2,  0) {$\covm$};
		\node (t7) [fill=gray!30, draw, text=black, circle,inner sep=1.0pt] at (1.65*\w, -0.3*\h) {$\theta_7$};
		
		\node[textnode] (d) at (\w*2,  1.7*\h) {$\dd$};
		\node (t10) [fill=gray!30, draw, text=black, circle,inner sep=1.0pt] at (2.32*\w, 1.8*\h) {$\theta_{10}$};
		\node[textnode] (cofM) at (\w*2,  \h*3) {$\cdot$};

		\node[textnode] (nonN) at (\w*3,  0) {$\nonn$};
		\node (t11) [fill=gray!30, draw, text=black, circle,inner sep=1.0pt] at (2.85*\w, -0.3*\h) {$\theta_{11}$};
		
		\node[textnode] (cofN) at (\w*3,  \h*3) {$\cofn$};
		\node (t12) [fill=gray!30, draw, text=black, circle,inner sep=1.0pt] at (3.25*\w, 2.2*\h) {$\theta_{12}$};
		
		\node[textnode] (aleph1) at (-\w,  0) {$\aleph_1$};
		\node[textnode] (c) at (\w*4,  \h*3) {$2^{\aleph_0}$};
		\node (t10) [fill=gray!30, draw, text=black, circle,inner sep=1.0pt] at (3.67*\w, 3.4*\h) {$\theta_\cc$};
		
		\node[textnode] (e) at (0.5*\w,  1.7*\h) {$\mathfrak{e}$}; 
		\node[textnode] (estar) at (\w,  1.7*\h) {$\ee^*$};
		\node (t4) [fill=gray!30, draw, text=black, circle,inner sep=1.0pt] at (0.15*\w, 1.7*\h) {$\theta_4$};

		\node[textnode] (pr) at (2.5*\w,  1.3*\h) {$\mathfrak{pr}$}; 
		\node[textnode] (prstar) at (\w*2,  1.3*\h) {$\mathfrak{pr}^*$};
		\node (t9) [fill=gray!30, draw, text=black, circle,inner sep=1.0pt] at (2.85*\w, 1.3*\h) {$\theta_9$};
		
		\node[textnode] (eubd) at (\w*0.2,  \h*2.1) {$\ee_{ubd}$};
		\node[textnode] (nonE) at (\w*0.75,  \h*2.55) {$\non(\mathcal{E})$};
		\node (t5) [fill=gray!30, draw, text=black, circle,inner sep=1.0pt] at (0.5*\w, 2.1*\h) {$\theta_5$};
		
		\node[textnode] (prubd) at (\w*2.8,  \h*0.9) {$\mathfrak{pr}_{ubd}$};
		\node[textnode] (covE) at (\w*2.25,  \h*0.45) {$\cov(\mathcal{E})$};
		\node (t8) [fill=gray!30, draw, text=black, circle,inner sep=1.0pt] at (2.47*\w, 0.9*\h) {$\theta_8$};

		\draw[->, edge] (addN) to (covN);
		\draw[->, edge] (addN) to (addM);
		\draw[->, edge] (covN) to (nonM);	
		\draw[->, edge] (addM) to (b);
		\draw[->, edge] (addM) to (covM);
		\draw[->, edge] (nonM) to (cofM);
		\draw[->, edge] (d) to (cofM);
		\draw[->, edge] (b) to (prstar);
		\draw[->, edge] (covM) to (nonN);
		\draw[->, edge] (cofM) to (cofN);
		\draw[->, edge] (nonN) to (cofN);
		\draw[->, edge] (aleph1) to (addN);
		\draw[->, edge] (cofN) to (c);
		
		\draw[->, edge] (e) to (covM);
		\draw[->, edge] (addN) to (e);
		
		\draw[->, edge] (covM) to (prstar);
		\draw[->, edge] (nonM) to (pr);
		\draw[->, edge] (pr) to (cofN);
		
		\draw[->, edge] (e) to (estar);
		\draw[->, edge] (b) to (estar);C
		\draw[->, edge] (estar) to (nonM);
		\draw[->, edge] (estar) to (d);
		\draw[->, edge] (e) to (eubd);
		
		\draw[->, edge] (prstar) to (d);
		\draw[->, edge] (prstar) to (pr);
		
		\draw[->, edge] (prstar) to (pr);
		\draw[->, edge] (prubd) to (pr);
		\
		
		\draw[->, edge] (eubd) to (nonE);
		\draw[->, edge] (addM) to (nonE);
		\draw[->, edge] (nonE) to (nonM);
		
		\draw[->, edge] (covE) to (prubd);
		\draw[->, edge] (covE) to (cofM);
		\draw[->, edge] (covM) to (covE);
		
		\draw[black,thick] (-0.5*\w,1.5*\h)--(3.5*\w,1.5*\h);
		\draw[black,thick] (1.5*\w,-0.5*\h)--(1.5*\w,3.5*\h);
		
		\draw[black,thick] (-0.5*\w,-0.5*\h)--(-0.5*\w,3.5*\h);
		\draw[black,thick] (3.5*\w,-0.5*\h)--(3.5*\w,3.5*\h);
		
		\draw[black,thick] (0.5*\w,-0.5*\h)--(0.5*\w,1.5*\h);
		\draw[black,thick] (2.5*\w,1.5*\h)--(2.5*\w,3.5*\h);
		
		\draw[black,thick] (-0.1*\w,1.9*\h)--(1.5*\w,1.9*\h);
		\draw[black,thick] (-0.1*\w,2.7*\h)--(1.5*\w,2.7*\h);
		\draw[black,thick] (-0.1*\w,1.5*\h)--(-0.1*\w,2.7*\h);
		\draw[black,thick] (0.5*\w,2.7*\h)--(0.5*\w,3.5*\h);
		
		\draw[black,thick] (3.1*\w,1.1*\h)--(1.5*\w,1.1*\h);
		\draw[black,thick] (3.1*\w,0.3*\h)--(1.5*\w,0.3*\h);
		\draw[black,thick] (3.1*\w,1.5*\h)--(3.1*\w,0.3*\h);
		\draw[black,thick] (2.5*\w,0.3*\h)--(2.5*\w,-0.5*\h);

		\draw[->, edge] (nonE) to (nonN);
		\draw[->, edge] (covN) to (covE);

	\end{tikzpicture}
	\caption{Constellation of $\p^7_\mathrm{fin}$, if the construction is possible.}\label{fig_p7}
\end{figure}


\begin{thebibliography}{GKMS21}
	\bibitem[Bar10]{BartInv}
	Tomek Bartoszynski. Invariants of measure and category. In {\em Handbook of set theory. {V}ols. 1, 2, 3}, pages 491--555.
	Springer, Dordrecht, 2010.
	
	
%
	
	\bibitem[BCM23]{BCMseparating}
	J\"{o}rg Brendle, Miguel~A. Cardona, and Diego~A. Mej\'{\i}a. Separating cardinal characteristics of the strong measure zero ideal. Preprint, {\em arXiv:2309.01931}, 2023.
	
	\bibitem[BJ95]{BJ95}
	Tomek Bartoszy\'{n}ski and Haim Judah.
	{\em Set theory. On the structure of the real line}.
	A K Peters, Ltd., Wellesley, MA, 1995.
	
	
	\bibitem[Bre91]{Bre91}
	J\"{o}rg Brendle. Larger cardinals in Cicho{\'n}'s diagram. {\em The Journal of Symbolic Logic}, 56(3):795-810, 1991.
	
	
	\bibitem[BS96]{BS_E_and_P_2}
	J\"{o}rg Brendle and Saharon Shelah. Evasion and prediction {I}{I}.
	{\em Journal of the London Mathematical Society}, 53(1):19-27, 1996.
	
	\bibitem[CM22]{forcing_constellations}
	Miguel~A. Cardona and Diego~A. Mej\'{\i}a. Forcing constellations of {C}icho\'{n}'s diagram by using the {T}ukey order. Preprint, arXiv:2203.00615, 2022.
	
	\bibitem[CMU23]{Car23RIMS}
	Miguel~A. Cardona, Diego~A. Mej\'{\i}a, and Andr{\'e}s Felipe Uribe-Zapata.
	Controlling the uniformity of the ideal generated by the $F_\sigma$ measure zero subsets of the reals.
	RIMS Set Theory Workshop 2023,
	\url{https://tenasaku.com/RIMS2023/slides/cardona-rims2023.pdf},
	2023.
	%
	%
	\bibitem[FM21]{FM21}
	Saka\'{e} Fuchino and Diego~A. Mej\'{\i}a.
	Variations of the negation of {R}iis' axiom.
	First Brazil--Colombia Meeting in Logic, 2021.
	
	\bibitem[GKMS21]{GKMS_CM_and_evasion}
	Martin Goldstern, Jakob Kellner, Diego~A. Mej\'{\i}a, and Saharon Shelah.
	\newblock Adding the evasion number to Cicho\'n's Maximum.
	\newblock XVI International Luminy Workshop in Set Theory,
	\newblock \url{https://dmg.tuwien.ac.at/kellner/2021_Luminy_talk.pdf}, \newblock 2021.
	
	\bibitem[GKMS22]{GKMS}
	Martin Goldstern, Jakob Kellner, Diego~A. Mej\'{\i}a, and Saharon Shelah. Cicho\'{n}'s maximum without large cardinals. {\em J. Eur. Math. Soc. (JEMS)}, 24(11):3951--3967, 2022.
	
	\bibitem[GKS19]{GKS}
	Martin Goldstern, Jakob Kellner, and Saharon Shelah. Cicho\'{n}'s maximum.
	{\em Annals of Mathematics (2)}, 190(1):113--143, 2019.
	
	\bibitem[GMS16]{GMS16}
	Martin Goldstern, Diego~A. Mej\'{\i}a, and Saharon Shelah.
	The left side of Cicho\'{n}'s diagram.
	{\em Proceedings of the American Mathematical Society}, 144(9):4025-4042, 2016.
	
	\bibitem[JS90]{JS90}
	Haim Judah and Saharon Shelah. The Kunen-Miller chart (Lebesgue measure, the Baire property, Laver reals and preservation theorems for forcing).
	{\em The Journal of Symbolic Logic}, 55(3):909-927, 1990.
	
	\bibitem[Kam89]{Kam89}
	Anastasis Kamburelis.
	Iterations of Boolean algebras with measure.
	{\em Archive for Mathematical Logic}, 29:21-28, 1989.
	
	\bibitem[KST19]{KST}
	Jakob Kellner, Saharon Shelah, and Anda Ramona T{\u{a}}nasie.
	Another ordering of the ten cardinal characteristics in {C}icho\'{n}'s diagram.
	{\em Commentationes Mathematicae Universitatis Carolinae}, 60(1):61-95, 2019.
	
	\bibitem[MCU23]{Mej_Two_FAM}
	Diego~A. Mej\'{\i}a, Miguel~A. Cardona, and Andr{\'e}s Felipe Uribe-Zapata.
	Two-dimensional iterations with fam-limits.
	RIMS Set Theory Workshop 2023,
	\url{https://tenasaku.com/RIMS2023/slides/mejia-20231027RIMS.pdf},
	2023.
	
	\bibitem[Mej13]{Mej13}
	Diego~A. Mej\'{\i}a
	Matrix iterations and {C}icho\'n's diagram.
	{\em Archive for Mathematical Logic}, 52(3-4):261-278, 2013.
	
	
	\bibitem[Mej19]{Mej19}
	Diego~A. Mej\'{\i}a
	Matrix iterations with vertical support restrictions.
	In {\em InProceedings of the 14th and 15th Asian Logic Conferences}, pages 213–248. World Sci. Publ., Hackensack, NJ, 2019.
	
	%
	\bibitem[She00]{She00}
	Saharon Shelah.
	Covering of the null ideal may have countable cofinality.
	{\em Fundamenta Mathematicae}, 166(1-2):109-136, 2000.
	
	\bibitem[Uri23]{Uri}
	Andr{\'e}s Felipe Uribe-Zapata.
	Iterated forcing with finitely additive measures: applications of probability to forcing theory.
	Master's thesis, Universidad Nacional de Colombia, sede Medell\'in,
	2023.
	
	
	\bibitem[Yam24]{Yam24}
	Takashi Yamazoe. Cicho\'n's maximum with evasion number. Preprint, {\em arXiv:2401.14600}, 2024.
\end{thebibliography}
\end{document}